\documentclass[12pt,reqno,a4paper]{amsart}
\usepackage{
    amsmath,  amsfonts, amssymb,  amsthm,   amscd,
    gensymb,  graphicx, comment,  etoolbox, url,
    booktabs, stackrel, mathtools,enumitem, mathdots,  microtype, lmodern,    mathrsfs, graphicx, tikz,  longtable,tabularx, float, tikz, pst-node, tikz-cd, multirow, tabularx, amscd,  bm, array, makecell, diagbox, booktabs,ragged2e, caption, subcaption }
\usepackage{makecell,slashbox}
\usepackage{xcolor}

\usepackage[utf8]{inputenc}
\usepackage{microtype, fullpage, wrapfig,textcomp,mathrsfs,csquotes,fbb}
\usepackage[colorlinks=true, linkcolor=blue, citecolor=blue, urlcolor=blue, breaklinks=true]{hyperref}
\usepackage[capitalise]{cleveref}
\usepackage{todonotes}
\usetikzlibrary{positioning}
\usetikzlibrary{shapes,arrows.meta,calc}
\usetikzlibrary{arrows}

\newcommand{\RomanNumeralCaps}[1]
    {\MakeUppercase{\romannumeral #1}}
\newtheorem{theorem}{Theorem}[section]

\newtheorem{corollary}[theorem] {Corollary}
\newtheorem{definition}[theorem]{Definition}
\newtheorem{example}[theorem]{Example}
\newtheorem{lemma}[theorem]{Lemma}

\newtheorem{proposition}[theorem]{Proposition}
\newtheorem{remark}[theorem]{Remark}

\newcommand\R{\mathbb{R}}
\newcommand\Z{\mathbb{Z}}
\newcommand\C{\mathbb{C}}

\newcommand{\TC}{\mathrm{TC}}

\newcommand{\ct}{\mathrm{cat}}

\newcommand{\sct}{\mathrm{secat}}

\newcolumntype{x}[1]{>{\centering\arraybackslash}p{#1}}

\begin{document}
\title[]{Equivariant parametrized topological complexity}

\author[N. Daundkar]{Navnath Daundkar}
\address{Department of Mathematics, Indian Institute of Science Education and Research Pune, India.}
\email{navnath.daundkar@acads.iiserpune.ac.in}

\thanks{}

\begin{abstract} 
In this paper, we define and study an  equivariant analogue of Cohen, Farber and Weinberger's parametrized topological complexity.  We show that several results in the non-equivariant case can be extended to the equivariant case. For example, we establish the fibrewise equivariant homotopy invariance of the sequential equivariant parametrized topological complexity. We obtain several bounds on sequential equivariant topological complexity involving the equivariant category. We also obtain the cohomological lower bound and the dimension-connectivity upper bound on the sequential equivariant parametrized topological complexity.
In the end, we use these results to compute the sequential equivariant parametrized topological complexity of equivariant Fadell-Neuwirth fibrations and some equivariant fibrations involving generalized projective product spaces. 
\end{abstract}
\keywords{Parametrized topological complexity, equivariant sectional category}
\subjclass[2020]{55M30, 55S40, 55R10, 55R91}
\maketitle

\section{Introduction}\label{sec:intro}
A solution to the motion planning problem in robotics is given by a motion planning algorithm, which is a function that takes as inputs a pair of configurations of a mechanical system and produces a continuous path connecting these two configurations as an output. More specifically, consider $X$ to be the configuration space of a mechanical system and $X^I$ to be the free path space of $X$ with a compact open topology.
The free path space fibration $\pi:X^I\to X\times X$ is defined by $\pi(\gamma)=(\gamma(0),\gamma(1))$.
A \emph{motion planning algorithm} is then a section of a free path space fibration.

Farber \cite{FarberTC} introduced the concept of topological complexity to analyze the difficulty of computing a motion planning algorithm for the configuration space $X$ of a mechanical system.
The \emph{topological complexity} of a space $X$, denoted by $\TC(X)$, is defined as the smallest positive integer $k$ for which $X\times X$ can be covered by open sets $\{U_1,\dots, U_k\}$ such that each $U_i$ admits a continuous local section of $\pi$.
Farber showed that a motion planning algorithm in the configuration space cannot be continuous unless it is contractible. Therefore, for non-contractible spaces, the topological complexity is at least two.
The numerical invariant topological complexity 
has been extensively studied over the past two decades.

To generalise topological complexity, Rudyak introduced its sequential analogue in \cite{RUD2010}. This models the motion planning problem for robots that need to visit several places in between their initial and final stages while performing their tasks. 
We briefly recall the definition.
Consider the  generalized  free path space fibration $\pi_n: X^I\to X^n$ of path-connected spaces $X$
 \begin{equation}\label{eq:general-free-path-fib}
   \pi_n(\gamma)=\bigg(\gamma(0), \gamma(\frac{1}{n-1}),\dots,\gamma(\frac{n-2}{n-1}),\gamma(1)\bigg).    \end{equation}
The \emph{sequential topological complexity} of $X$ is the smallest positive integer $k$ for which $X^n$ is covered by open sets $\{U_1,\dots, U_k\}$, such that each $U_i$ admits a continuous local section of $\pi_n$.
Note that by definition, it follows that $\TC_2(X)=\TC(X)$. 
Several properties of topological complexity have been generalized and introduced in their symmetrized version in \cite{gonzalezhighertc}.

There is an old invariant called LS-category, a close relative of sequential topological complexity, which was introduced by Lusternik and Schnirelmann in \cite{LScat}. The \emph{LS-category} of a space $X$ is denoted by $\ct(X)$, which is the least number of open subsets which cover $X$ such that the inclusion on each open set is nullhomotopic.
The authors of \cite{gonzalezhighertc} prove the following famous inequality
\[\mathrm{cat}(X^{k-1})\leq \TC_k(X)\leq \mathrm{cat}(X^k).\]

The notion of parametrized topological complexity was introduced by Cohen, Farber and Weinberger in \cite{PTC}, and the corresponding sequential analogue was introduced by Farber and Paul in \cite{SequentialPTC}.   
These parametrized motion planning algorithms offer enhanced universality and flexibility, enabling effective operation across various scenarios incorporating external conditions. These conditions are considered parameters and integral components of the algorithm's input.
A parametrized motion planning algorithm takes as inputs a pair of configurations living under the same external conditions and produces a continuous motion of the system, which remains constant under the external conditions.

We now define sequential parametrized topological complexity in brief. For a fibration $p: E \to B$, 
consider the fibre product $E^n_B$, the space of all
$n$-tuples of points in $E$ all of which lie in a common fibre of $p$. The space $E^I_B$ denotes the space of all paths in $E$ with image in a
single fibre. Note that the fibration defined in \eqref{eq:general-free-path-fib} restricts to the subspace $E^I_B$ of $E^I$. We denote this fibration by $\Pi_n:E^I_B\to E^n_B$.
The \emph{sequential parametrized topological complexity} of a fibration $p : E \to B$ denoted by $\TC_n[p:E\to B]$  is the smallest integer $k$ such that there is an open cover $\{U_1,\dots, U_k\}$ of $E^n_B$, where each open set $U_i$ admits a continuous section of $\Pi_n$. Note that $\TC_2[p:E\to B]$ is known as the \emph{parametrized topological complexity} of Cohen, Farber and Weinberger. The reader is referred to \cite{PTC}, \cite{PTCcolfree}, \cite{SequentialPTC}, \cite{ptcspherebundles} for several interesting results related to (sequential) parametrized topological complexity. We also mention that the notion of parametrized topological complexity of fibrations is extended to fibrewise spaces in \cite{fibrewise} by Garc\'{\i}a-Calcines.

The invariants we have discussed so far are special cases of a more general notion: the Schwarz genus of a fibration. Schwarz introduced and studied this notion in \cite{Sva}. Later, Bernstein and Ganea \cite{secat} extended and studied this notion for any map. 
Let's now define this notion.
The \emph{sectional category} of a map $p:E\to B$, denoted $\sct(p)$, is the smallest integer $k$ for which $B$ can be covered by $k$ open sets $W_1,\dots,W_k$, where each $W_i$ admits a continuous homotopy section of $p$. 
If $p: E\to B$ is a fibration, then $\sct(p)$ coincides with another invariant called the \emph{Schwarz genus} of a fibration, denoted $\mathrm{gen}(p)$ (see \cite{Sva}). 
For example, $\TC_n(X)=\sct(\pi_n)$, $\ct(X)=\sct(i: \{\ast\}\hookrightarrow{} X)$ and $\TC_n[p:E\to B]=\sct(\Pi_n)$.
The reader is referred to \cite{James} for a more general overview of the sectional category. 

Generalizing the concept of sectional category,
Colman and Grant \cite{EqTC} first introduced the corresponding equivariant analogue. The equivariant sectional category of a $G$-map $p:E\to B$ between $G$-spaces is denoted by $\sct_G(p)$ (see \Cref{def:eqsecat}). Colman and Grant extended some of the classical results from \cite{Sva} to the equivariant settings.
It can be observed that the (generalized) free path space fibration is a $G$-map (in fact, a $G$-fibration) (see \Cref{sec:eq-secat}). 
Colman and Grant introduced the notion of equivariant topological complexity as a consequence of the equivariant sectional category.
The sequential analogue of this concept was introduced by Bayeh and Sarkar in \cite{byehsarkareqtcn}.  
They established several interesting relationships between the equivariant category \cite{Eqlscategory} and sequential equivariant topological complexity. Additionally, they obtained new bounds on its corresponding non-equivariant counterpart.

By a (Serre) $G$-fibration we will mean a $G$-map $p : E\to B$ having the $G$-
homotopy lifting property with respect to all $G$-spaces ($G$-CW complexes).
One can observe that for a $G$-fibration $p:E\to B$, the corresponding fibre product $E^n_B$ admits a  $G$-action. Moreover the space $E^I_B$ is also $G$-invariant and the fibration $\Pi_n:E^I_B\to E^n_B$ is a $G$-map (see \Cref{sec:eqptc} for details). Therefore, expecting an equivariant analogue of parametrized topological complexity is natural.
In \Cref{sec:eqptc}, we define the sequential equivariant parametrized topological complexity as the equivariant sectional category of a $G$-map $\Pi_n$. We denote this newly defined notion by $\TC_{G,n}[p:E\to B]:=\sct_G(\Pi_n)$.    
This notion generalizes sequential parametrized topological complexity. We generalize several results from the theory of sequential parametrized topological complexity. We define the fibrewise $G$-homotopy equivalent fibrations (see \Cref{def: G-fib hteq}) and  establish the fibrewise $G$-homotopy invariance of $\TC_{G,n}[p;E\to B]$ (see \Cref{prop:fib-G-htp-inv}). Later, in \Cref{subsec:bounds}, we obtain several bounds on the sequential equivariant parametrized topological complexity involving the equivariant category of fibre product. We also obtain the cohomological lower bound \Cref{thm:coho-lb} and the dimension-connectivity upper bound in \Cref{thm:dim-conn-ub}.

\section{Sequential parametrized topological complexity}

Cohen, Farber and Weinberger introduced the notion of parametrized topological complexity in \cite{PTC} ,\cite{PTCcolfree}. 
This section briefly recalls some basic results related to this notion.

For a Hurewicz fibration $p:E\to B$, consider a subspace of the path space $E^I$, defined as follows:
\[E^I_B:=\{\gamma\in E^I \mid p\circ \gamma(t)=b ~\text{for some}~ b\in B ~\text{and for all}~ t\in[0,1] \}.\]
The fibre product corresponding to $p:E\to B$ is defined by 
\[E^n_B:=\{(e_1,\dots,e_n)\in E^n \mid p(e_i)=p(e_j) ~\text{ for }~ 1\leq i,j\leq n\}.\]
Define a map $\Pi_n: E^I_B\to E^n_B$ by 
\begin{equation}\label{eq: Pin}
\Pi_n(\gamma)= \bigg(\gamma(0), \gamma(\frac{1}{n-1}),\dots,\gamma(\frac{n-2}{n-1}),\gamma(1)\bigg) .   
\end{equation}
It follows from the appendix of \cite{PTCcolfree} that $\Pi_n$ is a Hurewicz fibration.  A section of $\Pi_n$ is called a \emph{sequential parametrized motion planning algorithm}. 

In their recent paper \cite{SequentialPTC}, Farber and Paul have established a theory for the sequential parametrized topological complexity, generalizing several results of the original parametrized topological complexity.
 
\begin{definition}
Let $p:E\to B$ be a fibration. The $n$-th sequential parametrized topological complexity of $p$ is denoted by $\TC_n[p:E\to B]$, and defined as \[\TC_n[p:E\to B]:=\sct(\Pi_n).\]   
\end{definition}

It is observed in \cite{SequentialPTC} that for a pullback fibration $q:E'\to B'$ along a map $f:B'\to B$ of $p:E\to B$, we have \[\TC_n[q:E'\to B']\leq \TC_n[p:E\to B].\]
In particular, for a fibration $p:E\to B$ with fibre $F$, we have 
\begin{equation}\label{eq:tcn-fib-leq-ptc}
    \TC_n(F)\leq \TC_n[p:E\to B].
\end{equation}
Various bounds have been established to estimate the parametrized topological complexity in \cite{PTC}, and later, these were generalized in the sequential setting in \cite{SequentialPTC}. For example, the cohomological lower bound for the parametrized topological complexity is given in \cite{PTC}, and its sequential analogue is stated in \cite{SequentialPTC}. We state the sequential version of the cohomological lower bound.

\begin{theorem}[{\cite[Proposition 6.3]{SequentialPTC}}]
Let $p:E\to B$ be a fibration, $\triangle: E\to E^n_B$ be the diagonal map and $\triangle^{\ast}:H^{\ast}(E^n_B;R)\to H^{\ast}(E;R)$ be the corresponding induced homomorphism. Suppose there exists cohomology classes $u_i\in \ker(\triangle^{\ast}) $ for $1\leq i\leq r$ such that their cup product  $u_1\cup \dots \cup u_r$ does not vanish. Then \[\TC_n[p:E\to B]\geq r+1.\]   
\end{theorem}

In \cite{SequentialPTC}, Farber and Paul established a dimension-connectivity upper bound on the sequential parametrized topological complexity, generalizing the base case proved in  \cite[Proposition 7.1]{PTC}.
\begin{theorem}\cite[Propositon 6.1]{SequentialPTC}\label{thm:ptc-dim-conn-ub}
 Let $p:E\to B$ be a fibration with fibre $F$ such that spaces $E$, $B$ and $F$ are CW-complexes. Suppose that for $m\geq 0$,  $F$ is $m$-connected. Then 
 \[\TC_n[p:E\to B]< \frac{n\mathrm{dim}(F)+\mathrm{dim}(B)+1}{m+1}+1.\]
\end{theorem}

We now present the computations of the sequential parametrized topological complexity for some fibrations. 
Sarkar and Zvengrowski \cite{sarkargpps} introduced a class of topological spaces called generalized projective product spaces. We will now recall the definition of these spaces.
Let $M$ and $N$ be CW-complexes with involutions $\tau \colon M \to M$ and $\sigma \colon N \to N$, where $\sigma$ is fixed-point-free. The following identification space is called \emph{generalized projective product space}
\begin{equation}\label{eq:gpps}
X(M, N) :=\displaystyle\frac{M \times N}{(x,y)\sim (\tau(x), \sigma(y))}.
\end{equation}
Note that this class of manifolds includes all projective product spaces \cite{Davis} and Dold manifolds \cite{Dold}.
The author and Soumen Sarkar studied the LS-category and topological complexity of generalized projective product spaces in \cite{DaundSarkargpps}.

\begin{example}
{\normalfont 
Observe that  there is a fibre bundle \begin{equation}\label{eq:gpps-fib-bundle}
    M \xhookrightarrow{} X(M, N) \stackrel{p}\longrightarrow N/\left<\sigma \right>
\end{equation}
defined by $p([(x,y)])=[y]$, where $N/\left<\sigma \right>$ is the orbit space of the group $\left<\sigma \right>$ action on $N$ induced by the involution $\sigma$. 
 
It follows from \Cref{thm:ptc-dim-conn-ub} that \[\TC_n[p:X(M,N)\to N/\left<\sigma \right>]< n\dim(M)+\dim(N)+2.\] In particular, \[\TC_n[p:X(M,N)\to N/\left<\sigma \right>]\leq n\dim(M)+\dim(N)+1.\]

Also, from \eqref{eq:tcn-fib-leq-ptc}, we have the following lower bound: \[\TC_n(M)\leq \TC_n[p:X(M,N)\to N/\left<\sigma \right>].\]
Observe that if $N=S^0$ with free involution, then $N/\left<\sigma \right>$ is a singleton and $X(M,N)\cong M$. Thus, $\TC_n[p:X(M,S^0)\to S^0/\left<\sigma \right>]=\TC_n(M)$.

Note that if $N=S^1$ with the antipodal involution and $\TC_n(M)=n\dim(M)+1$, then \[n\dim(M)+1\leq\TC_n[p:X(M,N)\to N/\left<\sigma \right>]\leq n\dim(M)+2.\]
One example of spaces with involutions having maximal sequential topological complexity are orientable surfaces of positive genus. An orientable surface of genus $g\geq 1$ can be embedded in $\R^3$ to admit antipodal involution. It was proved in \cite{seqtcsurfaces} that $\TC_n(\Sigma_g)=n\dim(\Sigma_g)+1=2n+1$.

In some cases, we can now compute the exact value of sequential parametrized topological complexity.
Suppose for $d\geq 2$, consider $M=S^d$ with an involution $\tau$ obtained by reflecting $S^d$ across the $X$-axis. Observe that $p:X(S^d,N)\to N/\left<\sigma \right>$ is a sphere bundle with structure group $\Z_2=\left<\tau \right>$. One can generalize \cite[Example 5.9]{EqTC} to obtain $\TC_{\Z_2,n}(S^d)=n+1$. Then it follows from \cite[Theorem 3.4]{ptcrelateinv} that \[\TC_n[p:X(S^d,N)\to N/\left<\sigma \right>]\leq \TC_{\Z_2,n}(S^d)=n+1.\]
From \eqref{eq:tcn-fib-leq-ptc} we have $\TC_n(S^d)\leq \TC_n[p:X(S^d,N)\to N/\left<\sigma \right>]$.
Therefore, if $d$ is even, we have $ \TC_n[p:X(S^d,N)\to N/\left<\sigma \right>]=n+1$. If $d$ is odd, then $\TC_n[p:X(S^d,N)\to N/\left<\sigma \right>]\in\{n, n+1\}$.
}
\end{example}

\begin{example}
{\normalfont 
In \cite{Milnor}, Milnor introduced a class of submanifolds of the products of real and complex projective spaces to define generators for the unoriented cobordism algebra.We define these manifolds as follows:
Let  $r$ and $s$ be integers such that $0\leq s\leq r$. 
A \emph{real Milnor manifold}, denoted by $\mathbb{R}M_{r,s}$, is the submanifold of the product $\mathbb{R}P^r \times \mathbb{R}P^s$ of real projective spaces. It is defined as follows:
\[\mathbb{R}M_{r,s}:=\Big\{\big([x_0: \dots: x_r], [y_0:\dots:y_s]\big) \in \mathbb{R}P^r \times \mathbb{R}P^s\mid x_0y_0 +\hspace{1mm}\cdots \hspace{1mm}+ x_sy_s=0\Big\}.\] 
The real Milnor manifold $\mathbb{R}M_{r,s}$ is an $(s+r-1)$-dimensional closed, smooth manifold.
There is a fiber bundle:
\begin{equation}\label{eq: fbRHrs}
 \mathbb{R}P^{r-1} \stackrel{i}{\hookrightarrow} \mathbb{R}M_{r,s} \stackrel{p}{\longrightarrow} \mathbb{R}P^{s}.
\end{equation}
\emph{A complex Milnor manifold}, denoted by $\C M_{r,s}$, can be defined in a  similar manner:
\[\C M_{r,s}:=\Big\{\big([z_0:\dots:z_r],[w_0:\dots:w_s]\big)\in \C P^r \times \C P^s\mid z_0\overline{w}_0+\cdots+z_s\overline{w}_s=0\Big\}.\] 
Note that $\C M_{r,s}$ is a closed, smooth manifold of dimension $2(s+r-1)$.
Just as in the real case, $\C M_{r,s}$ is the total space of the fiber bundle:
\begin{equation}\label{eq: fbCHrs}
 \C P^{r-1} \stackrel{i}{\hookrightarrow} \C M_{r,s} \stackrel{q}{\longrightarrow} \C P^{s}.
\end{equation}
The author and B. Singh \cite{daundkarsingh} studied sequential topological complexity of $\R M_{r,s}$ and  $\C M_{r,s}$, obtaining several sharp bounds and computing exact values in various cases.

Note that, using  \eqref{eq:tcn-fib-leq-ptc} and \cite[Proposition 6.1] {SequentialPTC}, we obtain 
\[\TC_n(\R P^{r-1})\leq\TC_n[p:\mathbb{R}M_{r,s} \stackrel{}{\longrightarrow} \mathbb{R}P^{s}]\leq n(r-1)+s+1.\]
The topological complexity of projective spaces has been studied in \cite{FarberTCproj} and the corresponding sequential analogue is discussed in \cite{tcnproj}.

If $r=2^m+1$, then it is known that $\TC_n(\R P^{r-1})=n(r-1)$. Thus, for $r=2^m+1$, we have
\[n(r-1)\leq\TC_n[p:\mathbb{R}M_{r,s} \stackrel{}{\longrightarrow} \mathbb{R}P^{s}]\leq n(r-1)+s+1.\]

Since $\C M_{r,s}$ is simply connected, applying \Cref{thm:ptc-dim-conn-ub} gives us
\[\TC_n[q:\C M_{r,s} \stackrel{}{\longrightarrow} \C P^{s}]\leq n(r-1)+s+1.\]
Since $\TC_n(\C P^{r-1})=n(r-1)+1$, using \eqref{eq:ptc-fib-ineq} we obtain
\[n(r-1)+1=\TC_n(\C P^{r-1})\leq \TC_n[q:\C M_{r,s} \stackrel{}{\longrightarrow} \C P^{s}].\]
Thus, if $s=1$, then \[n(r-1)+1\leq \TC_{n}[\C M_{r,s} \stackrel{q}{\longrightarrow} \C P^1]\leq  n(r-1)+2.\]}
\end{example}

\section{Equivariant sectional category}\label{sec:eq-secat}
The sectional category of a Hurewicz fibration was introduced by \v{S}varc in  \cite{Sva}. Colman and Grant \cite{EqTC}  generalized this notion and introduced the equivariant sectional category. Let $G$ be a compact Lie group and $E, B$ be $G$-spaces such that $p\colon E \to B$ is a $G$-map.  The symbol $\simeq_G$ denotes the $G$-homotopy equivalence. In this paper, when we refer to a group $G$, we mean a compact Lie group.
\begin{definition}[{\cite[Definition 4.1]{EqTC}}\label{def:eqsecat}]
Let $p:E\to B$ be a $G$-map. 
The equivariant sectional category of $p$, denoted by $\sct_G(p)$, is the least positive integer $k$ such that there is a $G$-invariant open cover $\{U_1, \ldots, U_k\}$ of $B$ and $G$-maps $s_i \colon U_i \to E$ for $i=1, \ldots, k$, such that $p\circ s_i \simeq_G i_{U_i}$, where $i_{U_i} : U_i \hookrightarrow B$ is the inclusion map.
\end{definition}
If no such $k$ exists, we say $\sct_G(p)=\infty$.
If $G$ is a trivial group, then $\sct_G(p)$ is called the sectional category of $p$, denoted by $\sct(p)$. 

Fadell \cite{Fadelleqcat} introduced the notion of $G$-equivariant LS-category for $G$-spaces, and it was later studied by 
Marzantowicz  in \cite{Eqlscategory},  Clapp and Puppe in \cite{Clapp-Puppe}, Colman \cite{Colmaneqcat}, and  Angel, Colman, Grant, and Oprea in \cite{Angel-Colman-Grant-Oprea-Moritainv-eqcat}. This homotopy invariant of a $G$-space $X$ is denoted by $\ct_G(X)$.
Before defining this notion, we recall the notion of $G$-categorical sets. An invariant open set $U$ of a $G$-space $X$ is called \emph{$G$-categorical} if the inclusion $i_{U}:U\hookrightarrow X$ is $G$-homotopic to the map that takes values in a single orbit. 

\begin{definition}
For a $G$-space $X$, the $G$-equivariant category $\ct_G(X)$ is defined as the least positive integer $r$ such that $X$ can be covered by $r$ $G$-categorical sets.  
\end{definition}

Let $H$ be a closed subgroup of $G$ and $X$ be a $G$-space. The $H$-fixed point set of $X$ is denoted by $X^H$ and defined as \[X^H:=\{x\in X \mid hx=x ~\text{ for all }~ h\in H\}.\]
\begin{definition}
    A $G$-space $X$ is said to be $G$-connected if, for any closed subgroup $H$ of $G$, the $H$-fixed point set $X^H$ is path-connected.
\end{definition}
It can be observed that  if $x\in X$ is fixed point of a $G$-action and if $X$ is a $G$-connected space, then for the inclusion map $i\colon \{x\} \to X$, $\sct_G(i\colon \{x\} \to X) = \ct_G(X)$; see \cite[Corollary 4.7]{EqTC}. Additionally, if $G$ is trivial, then $\ct_G(X)=\ct(X)$.

We now recall the notion of sequential (higher) equivariant topological complexity introduced by Bayeh and Sarkar in \cite{byehsarkareqtcn}. This notion generalizes the equivariant topological complexity introduced by Colman and Grant in \cite{EqTC}.

Let $X$ be a $G$-space. Observe that the path space admits $G$-action via $(g\cdot \gamma)(t):=g\cdot \gamma(t)$. The product $X^n$ is also a $G$-space with the diagonal $G$-action. One can check that the fibration \eqref{eq:general-free-path-fib} is a $G$-fibration. 
\begin{definition}
Let $X$ be a $G$-space. The sequential equivariant topological complexity is denoted by $\TC_{G,n}(X)$ is the least positive integer $k$ such that the space $X^n$ is covered by $k$-many $G$-invariant open sets with each of which having an $G$-equivariant section of $\pi_n$.  In other words, $\TC_{G,n}(X)=\sct_G(\pi_n)$.
\end{definition}

In the next section, we introduce sequential equivariant parametrized topological complexity, a special case of which is sequential equivariant topological complexity.

Several results from \cite{EqTC} are generalized by Bayeh and Sarkar in \cite{byehsarkareqtcn}.
The following result is a sequential analogue of \cite[Corollary 5.8]{EqTC}. We couldn't find this result in Bayeh and Sarkar's paper \cite{byehsarkareqtcn}.
\begin{proposition}\label{prop:eqcat-htc-ineq}
Let $X$ be a $G$-space and $H$ be a stabilizer of some $a\in X$.
Then \[\ct_H(X^{n-1})\leq \TC_{G,n}(X).\]
\end{proposition}
\begin{proof}
The proof is similar to that of \cite[Corollary 5.8]{EqTC}. 
Define a map $f:X^{n-1}\to X^n$ by $f(x_1,\dots,x_{n-1})=(a,x_1,\dots,x_{n-1})$.
Then note that $f$ is $H$-equivariant map. 
 Now consider the following pullback diagram 
\[ \begin{tikzcd}
Y \arrow{r}{} \arrow[swap]{d}{q} & X^I \arrow{d}{\pi_n} \\%
X^{n-1} \arrow{r}{f}&X^n\\%
& 
\end{tikzcd}
,\] where $Y=\{\gamma\in X^I 
\mid \gamma(0)=a\}$
Note that $Y$ admits $H$-action. In fact, it is $H$-contractible. Then from \cite[Proposition 4.6]{EqTC}  and 
from \cite[Proposition 4.3]{EqTC}, we get \[\ct_H(X^{n-1})\leq \sct_{H}(q)\leq \sct_H(q)\leq \sct_H(\pi_n)=\TC_{H,n}(X).\]
Now from \cite[Proposition 3.14]{byehsarkareqtcn} we have $\TC_{H,n}(X) \leq \TC_{G,n}(X)$. This proves the desired inequality.
\end{proof}

The following result is also missed out in \cite{byehsarkareqtcn}, which is a sequential analogue of \cite[Proposition 5.12]{EqTC}. 
\begin{proposition}\label{prop: eqseqtc-tpgroup}
    Let $A$ be a topological group and $G$ acts on $A$ via topological group homomorphisms such that $A$ is $G$-connected. Then $\TC_{G,n}(A)=\ct_G(A^{n-1})$.
\end{proposition}
\begin{proof}
 Since $G$ acts on $A$ via topological group homomorphism and on $A^{n-1}$ diagonally, the identity element $\bar{e}=(e,\dots,e)\in A^{n-1}$ becomes a fixed point of $G$-action on $A^{n-1}$. Therefore, from \Cref{prop:eqcat-htc-ineq}, we get $\ct_G(A^{n-1})\leq \TC_{G,n}(A)$. 
 
 We now prove the other inequality. Suppose $\ct_G(A^{n-1})=r$. Then there exist a cover of $r$-many $G$-categorical subsets of of $A^{n-1}$. Let such cover be $\{U_1,\dots,U_r\}$. Let $u\in U_i$. Then from \cite[Lemma 3.14]{EqTC} we get a $G$-homotopy $H_i:O(u)\times I\to A^{n-1}$ such that $H_i(x,0)=x$ and $H_i(x,1)\in O(\bar{e})=\{\bar{e}\}$ for each $1\leq i\leq r$. Note that since $U_i$'s are $G$-invariant, we have $U_i=\cup_{u\in U_i} O(u)$. Moreover, since orbits are either disjoint or coincide, we get homotopy $\tilde{H}_i:U_i\times I\to A^{n-1}$ such that $\tilde{H}_i(x,0)=x$ and $\tilde{H}_i(x,1)\in O(\bar{e})=\{\bar{e}\}$ for all $x\in U_i$ and $t\in [0,1]$ and for $1\leq i\leq r$.
 For each $1\leq i\leq r$, we define open $G$-invariant set $V_i=\{(x_1\dots,x_n)\in A^n \mid (x_1x_n^{-1},\dots,x_{n-1}x_n^{-1})\in U_i\}$ of $A^{n}$. Observe that each $V_i$ is open invariant subset of $A^{n-1}$ and the collection $\{V_i \mid 1\leq i\leq r\}$ forms a cover of $A^{n}$. 

Our task is now to define a $G$-section of $\pi_n:A^I\to A^n$ on each $V_i$ for $1\leq i\leq r$. Let $\bar{x}=(x_1,\dots,x_{n-1})\in U_i$. Then note that \[\tilde{H}_i(\bar{x},t)=(\gamma^i_{1\bar{x}}(t),\dots, \gamma^i_{k\bar{x}}(t)),\] where 
$\gamma^i_{j\bar{x}}(0)=x_i$ and  $\gamma^i_{j\bar{x}}(1)=e$.

Let $\bar{y}=(y_1,\dots,y_n)\in V_i$ and $\tilde{y}=(y_1y_n^{-1},\dots, y_{n-1}y_n^{-1})$. Then
 define $s_i:V_i\to A^I$ as follows
\begin{equation}\label{eq: section on Uij}
 s_i(\bar{y})(t):=\begin{cases}  
\gamma^i_{1\tilde{y}}*\bar{\gamma}^i_{2\tilde{y}}((n-1)t)y_n & t\in [0,\frac{1}{n-1}]\\
\gamma^i_{2\tilde{y}}*\bar{\gamma}^i_{3\tilde{y}}((n-1)t-1)y_n & t\in [\frac{1}{n-1},\frac{2}{n-1}]\\ 
\hspace{1cm}.&.\\
\hspace{1cm}.&.\\
\hspace{1cm}.&.\\
\gamma^i_{r\tilde{y}}*\bar{\gamma}^i_{r+1\tilde{y}}((n-1)t-r+1)y_n & t\in [\frac{r-1}{n-1},\frac{r}{n-1}]\\
\hspace{1cm}.&.\\
\hspace{1cm}.&.\\
\hspace{1cm}.&.\\
\gamma^i_{k-1\tilde{y}}*\bar{\gamma}^i_{n\tilde{y}}((n-1)t-(n-2))y_n& t\in [\frac{n-2}{n-1},1],
\end{cases}   
\end{equation}
where $*$ denotes the concatenation of paths.
One can see that $s_i$ is indeed a $G$-section of $\pi_n$.
This gives us $\TC_{G,n}(A)\leq \ct_G(A^{n-1})$.
\end{proof}

\section{Equivariant parametrized topological complexity}\label{sec:eqptc}
This section introduces the concept of sequential equivariant parametrized topological complexity and examines its properties. We define fibrewise $G$-homotopy equivalent fibrations and prove one of the important results of this section which states that the sequential $G$-equivariant parametrized topological complexity of
fibrewise $G$-homotopy equivalent fibrations coincides. Then, we obtain various bounds on sequential equivariant parametrized topological complexity, which generalize bounds on the equivariant topological complexity obtained by Colman and Grant in \cite{EqTC}.

Let $p:E\to B$ be a $G$-fibration, and let $E^n_B$ be the corresponding fibre product. Define a $G$-action on $E^n_B$ as $g\cdot(e_1,\dots,e_n)=(ge_1,\dots,ge_n)$. Since $p$ is $G$-map, we have $p(ge_i)=gp(g_i)=gp(e_j)=p(ge_j)$. Thus, this action is well-defined.
The $G$-action on $E^I_B$ is defined as follows. Let $\gamma\in E^I_B$. Then define $(g\cdot \gamma)(t)=g\gamma(t)$. Note that $p(g\cdot \gamma)=gp(\gamma)$ and since $p(\gamma)$ is constant, $p(g\cdot \gamma)$ is also constant. Therefore, $E^I_B$ also admits a $G$-action.
Observe that for $g\in G$ we have:
\begin{align*}
 \Pi_n(g\gamma)&= \bigg(g\gamma(0), g\gamma(\frac{1}{n-1}),\dots,g\gamma(\frac{n-2}{n-1}),g\gamma(1)\bigg)\\
&= g\cdot\bigg(\gamma(0), \gamma(\frac{1}{n-1}),\dots,\gamma(\frac{n-2}{n-1}),\gamma(1)\bigg) =g\cdot \Pi_n(\gamma).
\end{align*}
Thus, $\Pi_n$ is $G$-equivariant.

\begin{definition}\label{def: eqptc}
The sequential equivariant parametrized topological complexity of a $G$-fibration $p:E\to B$, denoted  by $\TC_{G,n}[p:E\to B]$, is defined as 
\[\TC_{G,n}[p:E\to B]:=\sct_G(\Pi_n:E^I_B\to E^n_B).\]    
\end{definition}

We recall that Colman and Grant observed in \cite{EqTC} that the fibration $\pi_n$ is a $G$-fibration (also it follows from \cite[Proposition 2.5]{Grantsymmtc}). This can be proved using an equivariant analogue of \cite[Theorem 2.8.2]{SpanierAT}.
Cohen, Farber and Weinberger \cite[Appendix]{PTC} have shown that the map $\Pi_n$ is a fibration.
We want to prove an equivariant analogue of their result.

In \cite{eqfibrations}, an equivariant analogue of lifting function was defined and equivariant analogue of \cite[Theorem 2.7.8]{SpanierAT} was proved (see  \cite[ Section 2.1 and Proposition 2.1.3]{eqfibrations}). We also refer the reader to \cite[Section 2]{Grantsymmtc}, \cite{eqliftingfn} and \cite{Wanereqfib}. 

Let $p:E\to B$ be a $G$-fibration. Let $X$ be a topological space with a trivial $G$-action. Consider the space
$E^X_B=\{f:X\to E \mid p\circ f ~\text{is constant}\}$ with the compact open topology, on which the $G$-action is defined as $(gf)(x)=g(f(x))$.
Let $(K,L)$ be a pair of CW-complexes with trivial $G$-action, and  $\Pi:E^K_B\to E^L_B$ be the restriction map defined by $\Pi(f)=f|_{L}$. Note that $\Pi$ is a $G$ map. 
\begin{proposition}\label{prop:Pi-g-fibration}
The map $\Pi$ is a $G$-fibration.    
\end{proposition}
\begin{proof}
Since $p:E\to B$ is a $G$-fibration, it follow that (see \cite[Proposition 2.1.3]{eqfibrations} or \cite[Lemma 2.2]{Grantsymmtc}), there exist a $G$-lifting function $\lambda: \bar{B}\to E^I$, where $\bar{B}=\{(e,\gamma)\in E\times B^I \mid p(e)=\gamma(0)\}$, such that $p\circ \lambda(e,\gamma)=\gamma$ and $\lambda(e,\gamma)(0)=e$.
Consider the trivial action of $G$ on $L\times I$. An element $h$ of $(E^L_B)^I$ can be considered as a map $h:L\times I\to E$ such that for any $t\in I$, the image $h(L\times \{t\})\subseteq E$ lies in a single fibre of $p$.
Note that, we need to construct a $G$-lifting function $\Lambda: \bar{E}^L_B\to (E^K_B)^I$ for $\Pi$, where
\[\bar{E}^L_B=\{(f,h)\in E^K_B\times (E^L_B)^I \mid f(x)=h(x,0) ~\text{for all}~ x\in L\}.\] That is, $\Lambda$ should satisfy the conditions: $\Lambda(f,h)(x,t)=h(x,t)$ for all $x\in L$ and $t\in I$, and $\Lambda(f,h)(x,0)=f(x)$ for $x\in K$. 
Note that $\bar{E}^L_B$ has the diagonal $G$-action.

Let $h\in (E^L_B)^I$. Define a path $\omega_h$ in $B$ by  $\omega_h(t)=p\circ h(x,t)$ for $x\in L$.
Now define $\tilde{\Lambda}:\bar{E}^L_B\to (E^K_B)^I$ by $\tilde{\Lambda}(f,h)(x,t)=\lambda(f(x), \omega_h)(t)$. Note that for $g\in G$, we have 
$\tilde{\Lambda}(gf, gh)(x, t) =\lambda (gf(x), \omega_{gh})(t) = \lambda(gf(x), g\omega_h)(t) = g\lambda(f(x), \omega_h)(t)$.
Thus, $\tilde{\Lambda}$ is a $G$-map.
Observe that,  $\tilde{\Lambda}(f,h)(x,0)=f(x)$ for $x\in K$.  Thus $\tilde{\Lambda}$ satisfies the second condition of a $G$-lifting function. However, the condition $\tilde{\Lambda}(f,h)(x,t)=h(x,t)$ for all $x\in L$ is not guaranteed. Thus, we need a modification to $\tilde{\Lambda}$.

Define $\alpha_h:L\times I\times I\to E$ by 
\[    \alpha_h(x,\tau,t)=\begin{cases}
    h(x,t),& 0\leq t\leq \tau,\\
    \lambda(h(x,\tau), \omega_h^{[\tau,1]}(\frac{t-\tau}{1-\tau})),& \tau\leq t\leq 1,
\end{cases}
\]
where $\omega_h^{[\tau,1]}$ is a path defined by $\omega_h^{[\tau,1]}(s)=\omega_h(\tau+(1-\tau)s)$ for $s\in I$. Note that for $x\in L$, we have $p\circ \alpha_h(x,\tau,t)=\omega_h(t)$, $\alpha_h(x,0,t)=\lambda(h(x,0),\omega_h)(t)=\tilde{\Lambda}(f,h)(x,t)$, and $\alpha_h(x,1,t)=h(x,t)$. 

We will now define \[\Lambda: \bar{E^L_B}\to (E^K_B)^I ~\text{by}~ \Lambda(f,h)(x,t)=(\tilde{\Lambda}(f,h)\cup \alpha_h)(\rho(x,1),t),\] where $\rho:K\times I\to K\times \{0\}\cup L\times I$ is a retraction. Here we have treated $\tilde{\Lambda}(f,g)$ as a function from $K\times \{0\}\times I\to E$. Since for $g\in G$, we have $g\alpha_h=\alpha_{gh}$ and $\tilde{\Lambda}$ is a $G$-map, we conclude that $\Lambda$ is a $G$-map. One can check that $\Lambda(f,h)(x,t)=h(x,t)$ for all $x\in L$ and $\Lambda(f,h)(x,0)=f(x)$ for $x\in K$. Thus, $\Lambda$ is a $G$-lifting function for $\Pi$.   
\end{proof}

\begin{corollary}\label{cor:Pin-G-fib}
The $G$-map $\Pi_n:E^I_B\to E^n_B$ is a $G$-fibration.
\end{corollary}
\begin{proof}
The proof follows from \Cref{prop:Pi-g-fibration}, by considering $K=I$ as a CW-complex with its subcomplex $L=\{0,\frac{1}{n-1},\dots,\frac{n-2}{n-1},1\}$.    
\end{proof}

The following proposition is a straightforward application of  \Cref{def: eqptc}.
\begin{proposition} \ 
\begin{enumerate}
\item Let $p:E\to B$ be a fibration, and let $G$ act trivially on $E$ and $B$. Then $\TC_{G,n}[p:E\to B]=\TC_n[p:E\to B]$. 

\item If $B= \{\ast\}$, then $\TC_{G,n}[p:E\to B]=\TC_{G,n}(E)$.
\item Suppose  a $G$-fibration $p:E\to B$ is trivial. Then $\TC_{G,n}[p:E\to B]=\TC_{G,n}(F)$. 
\end{enumerate}
\end{proposition}
\begin{proof}
Parts (1) and (2) are obvious. We prove part (3) here.  Let $F$ and $B$ be $G$-spaces and let $p: F\times B (=E)\to B$ be a trivial $G$-fibration. Then we have $E^n_B=F^n\times B$, $E^I_B=F^I\times B$ and  $\Pi_n: E^I_B\to E^n_B$ becomes $\pi_n\times id: F^I\times B\to F^n\times B$. It is easy to see that $\sct_G(\pi_n\times id)=\sct_G(\pi_n)$.  Thus we have \begin{equation}\label{eq: ptc-equals-tc-fibre}
  \TC_{G,n}[p:F\times B\to B]=\TC_{G,n}(F)  .
\end{equation}
\end{proof}

\begin{remark}
Note that if $G$ acts trivially on $F$ and $B$, then the conclusion of \eqref{eq: ptc-equals-tc-fibre} follows from \cite[Example 3.2]{SequentialPTC} and the $n=2$ case from \cite[Example 4.2]{PTC}.
\end{remark}

In the following result, we establish an equivariant analogue of  \cite[Lemma 2.4]{daundkar2024sequential}. We also show that if the fixed point set of base space of a $G$-fibration is non-empty, then the sequential equivariant topological complexity of fibre is dominated by the sequential equivariant parametrized topological complexity.

\begin{proposition}
Let $p:E\to B$ be a $G$-fibration and let $\tilde{p}: \tilde{E}\to \tilde{B}$ be a pullback fibration corresponding to a $G$-map $\tilde{B}\to B$.
Then, 
\begin{equation}\label{eq:ptc-pullback-ineq}
 \TC_{G,n}[\tilde{p}:\tilde{E}\to \tilde{B} ]\leq \TC_{G,n}[p:E\to B].   
\end{equation}
Moreover, if $E^G\neq \emptyset$, then $F$ admits $G$-action and 
\begin{equation}\label{eq:ptc-fib-ineq}
\TC_{G,n}(F)\leq \TC_{G,n}[p:E\to B].    
\end{equation}
\end{proposition}
\begin{proof}
We consider the diagram  of a pullback  of $G$-fibrations
\[ \begin{tikzcd}
\tilde{E} \arrow{r}{f'} \arrow[swap]{d}{\tilde{p}} & E \arrow{d}{p} \\%
\tilde{B} \arrow{r}{f}&B\\%
& 
\end{tikzcd}
.\]
Note that $f'$ induces a $G$-map $\tilde{f'}: \tilde{E}^n_{\tilde{B}}\to E^n_B$ defined by $\tilde{f'}(e_1,\dots,e_n)=(f'(e_1),\dots,f'(e_n))$. The map $\tilde{f'}$ is well defined because of the commutativity of the above diagram.
One can observe that the pullback of the $G$-fibration $\Pi_n: E^I_B\to E^n_B$ along $\tilde{f'}$ is isomorphic as a $G$-fibration to $\tilde{\Pi}_n:\tilde{E}^I_{\tilde{B}}\to \tilde{E}^n_{\tilde{B}}$. Thus, \[\TC_{G,n}[\tilde{p}:\tilde{E}\to \tilde{B}] =  \sct_G(\tilde{\Pi}_n)\leq \sct_G(\Pi_n)=\TC_{G,n}[p: E\to B].\]
Since $E^G\neq \emptyset$, we have $B^G\neq \emptyset$.
Then it is easy to see that, for $e\in E^G$, the fibre $F:=p^{-1}(p(e))$ admits a $G$-action. Then \eqref{eq:ptc-fib-ineq} follows from  \eqref{eq:ptc-pullback-ineq} by taking the pullback along the inclusion $\{b\}\hookrightarrow B$.
\end{proof}

In the following result, we show that, under certain conditions, having a continuous $G$-equivariant parametrized motion planning algorithm for a $G$-fibration is equivalent to the fibre being $G$-contractible.
\begin{proposition}\label{prop:cont-pmpa-Gcontr}
Let $p:E\to B$ be a $G$-fibration with  $\TC_{G,n}[p:E\to B]=1$. Suppose there exists $e\in E^G\neq \emptyset$ such that fibre $F:=p^{-1}(p(e))$ is $G$-connected. 
Then $F$ is $G$-contractible. 
Conversely, if $F$ is $G$-contractible and $E^n_B$ is a $G$-CW-complex, then $\TC_{G,n}[p:E\to B]=1$.
\end{proposition}
\begin{proof}
We have $\TC_{G,n}(F)=1$ from \eqref{eq:ptc-fib-ineq}. Then, from \Cref{prop:eqcat-htc-ineq}, we get that $\ct_G(F^{n-1})=1$. This forces $\ct_G(F)=1$. That is, $F$ is $G$-contractible.   

Note that the fibre of $\Pi_n$ is $(\Omega F)^{n-1}$. Since $e\in E^G\neq \emptyset$, the space $(\Omega F)^{n-1}$ is admits a $G$-action. Additionally,  $(\Omega F)^{n-1}$ is $G$-contractible because $F$ is $G$-contractible. Then, from the equivariant obstruction theory, there exist a continuous section of $\Pi_n:E^I_B\to E^n_B$. This proves the converse.
\end{proof}


In \cite{Dieck}, tom Dieck introduced the notion of principal $(G,\alpha, A)$-bundles, where $G$ and $A$ are topological groups and $G$ acts on $A$ via a    continuous homomorphism $\alpha:G\to \mathrm{Aut}(A)$. For more details, see \cite{eqprincipalbundles}.
We now define these bundles.

\begin{definition}[{\cite[Definition 1.1]{eqprincipalbundles}}]
Let $E$ and $B$ be left $G$-spaces. 
A locally trivial principal $A$-bundle $p:E\to B$ is called a principal $(G,\alpha,A)$-bundle if $p$ is $G$-equivariant and $g( x\cdot a)= gx\cdot ga$ for every $g\in G$, $a\in A$ and $x\in E$. 
\end{definition}
For such bundles we establish a relation between sequential equivariant parametrized topological complexity and the $G$-equivariant category of $A^{n-1}$. We observe that the latter coincides with the sequential $G$-equivariant topological complexity of $A$. The following result generalizes \cite[Proposition 3.3]{SequentialPTC}.

\begin{theorem}\label{thm:eqptc-eq-prin-bundle}
    Let $p:E\to B$ be a principal $(G,\alpha, A)$-bundle such that $A$ is $G$-connected. Then, \[\TC_{G,n}[p:E\to B]=\ct_G(A^{n-1})=\TC_{G,n}(A).\]
\end{theorem}
\begin{proof}
Note that $G$ acts on $A$ via a continuous homomorphism $\alpha:G\to \mathrm{Aut}(A)$. Therefore, for the identity $e\in A$,  we have $g\cdot e=\alpha(g)(e)=e$. Thus, $e$ is a fixed point of the $G$-action on $A$. 
This allows the path space $P_0(A)=\{\gamma\in A^I \mid \gamma(0)=e\}$ to admit a $G$-action. Moreover, we have assumed that $A$ is $G$-connected.
Thus, it follows from \cite[Corollary 4.7]{EqTC} that $\sct_G(q:P_0(A)\to A^{n-1})=\ct_G(A^{n-1})$, where $q$ is a $G$-map defined by $q(\gamma):=(\gamma(\frac{1}{n-1}),\dots,\gamma(\frac{n-2}{n-1}),\gamma(1))$.

We now use a similar idea which was used in \cite[Proposition 3.3]{SequentialPTC}. Consider the following commutative diagram:
\[ \begin{tikzcd}
P_0(A)\times E \arrow{r}{F} \arrow[swap]{d}{q\times id} & E^I_B \arrow{d}{\Pi_n} \\%
A^{n-1}\times E \arrow{r}{F'}&E^n_B\\%
& 
\end{tikzcd}
,\]
where $F$ and $F'$ are defined such that \[F(\gamma,x)(t)=x \cdot \gamma(t) ~\text{ and }~  F'(a_1,\dots,a_{n-1},x)=(x,x\cdot a_1,\dots,x\cdot a_{n-1}), ~\text{respectively}~.\]
Note that $F$ and $F'$ are homeomorphisms. We will verify that they are indeed $G$-maps.
Let $g\in G$. Then we have $F(g\gamma,gx)(t)= gx \cdot g\gamma(t)= g (x \cdot \gamma(t))=gF(\gamma,x)(t)$ and 
\begin{align*}
  F'(ga_1,\dots,ga_{n-1},gx)  &= (gx,  gx \cdot ga_1,\dots, gx\cdot ga_{n-1})\\
  &= (gx, g(x\cdot a_1),\dots, g(x\cdot a_{n-1}))\\
  &= g(x,  x\cdot a_1,\dots,x\cdot a_{n-1})\\
  &= gF'(a_1,\dots,a_{n-1},x).
\end{align*}
Thus, $F$ and $F'$ are $G$-equivariant homeomorphisms.
Consequently, we get 
\[\TC_{G,n}[p:E\to B]=\sct_G(\Pi_n)=\sct_G(q\times id)=\sct_G(q)=\ct_{G}(A^{n-1}).\]

The equality $\ct_G(A^{n-1})=\TC_{G,n}(A)$ follows from \Cref{prop: eqseqtc-tpgroup}.
\end{proof}

\subsection{Fibrewise $G$-homotopy invariance}
This subsection aims to achieve a fibrewise $G$-homotopy invariance of $G$-equivariant parametrized topological complexity.
We begin by proving a lemma which plays a crucial role in showing  that the sequential $G$-equivariant parametrized topological complexity of fibrewise $G$-homotopy equivalent fibrations coincide. 

\begin{lemma}\label{lem:Gsecat}  \
 \begin{enumerate}
\item Suppose $p:E\to B$ and $p':E'\to B$ are $G$-fibrations. If there is a continuous $G$-map $f:E\to E'$ which fits into the following 
commutative diagram \[\begin{tikzcd}
E \arrow[dr, "p"'] \arrow{rr}{f}
& & E' \arrow{dl}{p'} \\
& B 
\end{tikzcd},\] 
then 
$\sct_G(p')\leq \sct_G(p)$.

\item Suppose we have the following commutative diagram of $G$-maps:
\[ \begin{tikzcd}
E \arrow{r}{F} \arrow[swap]{d}{p} & E' \arrow{d}{p'} \arrow{r}{F'} & E \arrow{d}{p} \\%
B \arrow{r}{f}&B' \arrow{r}{f'} & B\\%
& 
\end{tikzcd}
,\] where $p:E\to B$ and $p':E'\to B'$ are $G$-fibrations.
Suppose $f'\circ f\simeq_G id_B$. Then \[\sct_G(p)\leq \sct_G(p').\]
\end{enumerate}   
\end{lemma}
\begin{proof}
For (1), consider a $G$-invariant open set $U$ of $B$ with a $G$-section $s_U$ of $p$. Note that $f\circ s_U$ is a $G$-map, and due to the commutativity of the above diagram, it defines a $G$-section of $p'$. Thus, we get the desired inequality in (1).

For (2), consider the pullback fibration $q:B\times_{B'} E'\to B$ of a $G$-fibration $p':E'\to B'$ along $f:B\to B'$. Then, from \cite[Proposition 4.3]{EqTC}, we have: 
\begin{equation}\label{eq: ineq-sct}
 \sct_{G}(q)\leq \sct_{G}(p').   
\end{equation}

Let $\tilde{F}=F'\circ q'$, where $q':B\times_{B'} E'\to E'$ be the projection onto the second factor.
Then, we have  $p\circ\tilde{F}((b,x))=p\circ F'\circ q'((b,x))=f'\circ p'(x)$. Since $p'(x)=f(b)$, it follows that  $p\circ\tilde{F}((b,x))=f'\circ f\circ q(b,x)$. In other words, $p\circ\tilde{F}=f'\circ f\circ q$.
Since $f'\circ f\simeq_G id_B$, we have $p\circ\tilde{F}\simeq_G q$.
Note that we have the following G-homotopy commutative diagram:
\[ \begin{tikzcd}
B\times_{B'}E' \arrow{r}{\tilde{F}} \arrow[swap]{d}{q} & E \arrow{d}{p} \\%
B \arrow{r}{id}&B\\%
& 
\end{tikzcd}
.\]
Now consider a $G$-homotopy $H_t:B\to B$ such that $H_0=f'\circ f$ and $H_1=id_B$. Using the $G$-homotopy lifting property, we obtain a $G$-homotopy $J_t:B\times_{B'} E'\to E$ such $J_0=\tilde{F}$, $p\circ J_t=H_t\circ q$ and $p\circ J_1=q$. Thus, from part (1), we get that $\sct_G(p)\leq \sct_G(q)$. We can now conclude the proof using \eqref{eq: ineq-sct}.
\end{proof}

We need the following definitions before introducing the equivariant analogue of fibrewise homotopy equivalent fibrations.
\begin{definition}
Let $p:E\to B$ and $q: E'\to B$ be $G$-fibrations. 
A fibrewise $G$-map from $p:E\to B$ to $q: E'\to B$ is a $G$-map $f:E\to E'$ such that $q\circ f=p$.    
\end{definition}

\begin{definition}
A fibrewise $G$-homotopy $F:E\times I\to E'$ is a $G$-map such that $q(F(-,t))=p$ for all $t\in I$. Thus, $F$ is a $G$-homotopy between fibrewise $G$-maps $F(-,0)$ and $F(-,1)$.
\end{definition}

\begin{definition}\label{def: G-fib hteq}   
Let $p:E\to B$ and $q: E'\to B$ be $G$-fibrations. Then $p$ and $q$ are said to be fibrewise $G$-homotopy equivalent if there exist fibrewise $G$-maps $f: E\to E'$ and $g:E'\to E$ such that there are fibrewise $G$-homotopies  from $f\circ g$ to $Id_{E'}$ and from $g\circ f$ to $Id_{E}$. 
\end{definition}


Farber and Paul \cite{SequentialPTC} show that the sequential parametrized topological complexity of fibrewise homotopy equivalent fibrations coincides.
We now prove the equivariant analogue of their result.
\begin{proposition}\label{prop:fib-G-htp-inv}
Suppose the $G$-fibrations $p:E\to B$ and $p:E'\to B$ are fibrewise $G$-homotopy equivalent. Then
\[\TC_{G,n}[p:E\to B]=\TC_{G,n}[p':E'\to B].\]
\end{proposition}
\begin{proof}
The commutative diagram of $G$-maps given in the hypothesis induces the following two commutative diagrams of $G$-maps:

\[ \begin{tikzcd}[column sep=large]
E^I_B \arrow{r}{F^I} \arrow[swap]{d}{\Pi_n} & E'^I_B \arrow{d}{\Pi'_n} \arrow{r}{F'^I} & E^I_B \arrow{d}{\Pi_n} \\%
E^n_B \arrow{r}{f^n}&E'^n_B \arrow{r}{f'^n} & E^n_B\\%
& 
\end{tikzcd}
\text{ and } 
 \qquad 
 \begin{tikzcd}[column sep=large]
E'^I_B \arrow{r}{F'^I} \arrow[swap]{d}{\Pi'_n} & E^I_B \arrow{d}{\Pi_n} \arrow{r}{F^I} & E'^I_B \arrow{d}{\Pi'_n} \\%
E'^n_B \arrow{r}{f'^n}&E^n_B \arrow{r}{f^n} & E'^n_B\\%
& 
\end{tikzcd}
,\]
where $F^I(\gamma)(t)=F(\gamma(t))$ and $F'^I$ defined similarly.
Since  $f'\circ f$ and $f\circ f'$  are fibrewise $G$-homotopic to the identity maps $id_E$ and  $id_{E'}$, respectively, we have $f'^n\circ f^n\simeq_G id_{E^n_B}$ and $f^n\circ f'^n\simeq_G id_{E'^n_B}$. Thus, we obtain the desired equality using part (2) of \Cref{lem:Gsecat}.
\end{proof}

\subsection{Bounds}\label{subsec:bounds}
In this subsection, we obtain various bounds on the sequential equivariant parametrized topological complexity. We note that the results in this subsection generalize the classical results of Farber \cite{FarberTC}, \cite{RUD2010} when we consider the trivial action of the group $G$, and the results of Colman-Grant \cite{EqTC} and Bayeh-Sarkar \cite{byehsarkareqtcn} when we consider the base space $B$ to be a singleton.


We now establish the upper bound on the sequential equivariant parametrized topological complexity in terms of the equivariant category, and consequently, in terms of the dimension of the fibre product. 
\begin{proposition}\label{prop:dim-ub-eqptc}
 Let $p:E\to B$ be a $G$-fibration with $E^G\neq \emptyset$ and $E^n_B$ is $G$-connected . Then 
 \begin{equation}\label{eq:htc-ctG-ineq}
 \TC_{G,n}[p:E\to B]\leq \ct_{G}(E^n_B)\leq \mathrm{dim}(E^n_B/G)+1.    
 \end{equation}
 Moreover, if $G$ acts locally smoothly on $E$ and $F$ is the fibre of a smooth fibre bundle $p:E\to B$, then  
 \begin{equation}\label{eq: dim-ub-eq-ptc}
  \TC_{G,n}[p:E\to B]\leq n\mathrm{dim}(F)+\mathrm{dim}(B)-\mathrm{dim}(P)+1,   
 \end{equation}
where $P$ is the principal orbit (an orbit of maximal dimension) of the $G$-action on $E^n_B$.
 \end{proposition}
\begin{proof}
One can observe that, since $E^G\neq \emptyset$, we have $(E^n_B)^G\neq \emptyset$. Consequently, $(E^I_B)^G\neq \emptyset$. Moreover, $E^n_B$ is $G$-connected, thus the first inequality of \eqref{eq:htc-ctG-ineq} follows from 
\cite[Proposition 4.4]{EqTC}. 
The right-most inequality of \eqref{eq:htc-ctG-ineq} follows from \cite[Corollary 1.12]{Eqlscategory}. 
It follows from \cite[Theorem \RomanNumeralCaps{4}.3.8]{Bredon} that  $\mathrm{dim}(E^n_B/G)=\mathrm{dim}(E^n_B)-\mathrm{dim}(P)$. Then, the inequality in \eqref{eq: dim-ub-eq-ptc} follows from the fact that $\mathrm{dim}(E^n_B)\leq n\mathrm{dim}(F)+\mathrm{dim}(B)$.
\end{proof}

\begin{corollary}
Suppose $E$ admits a locally smooth, non-trivial and semi-free action of $G$, with $p:E\to B$ being a $G$-fibration such that $E^G\neq \emptyset$ and $E^n_B$ is $G$-connected. Then
\begin{equation}\label{eq: eqhtc-gpact-ub}
 \TC_{G,n}[p:E\to B]\leq  n\mathrm{dim}(F)+\mathrm{dim}(B)-\mathrm{dim}(G)+1.   
\end{equation}

\end{corollary}
\begin{proof}
It follows that if the action is locally smooth, non-trivial, and semi-free, then the principal orbit's dimension coincides with the group's dimension. Thus, the conclusion follows from \Cref{prop:dim-ub-eqptc}.
\end{proof}

\begin{remark}
Consider $p:E\to B$ a $G$-fibration. When $B= \{\ast\}$, the inequality  \eqref{eq: eqhtc-gpact-ub} also appears in \cite[Proposition 3.1]{daundkarsingh}. Since $\TC_n(X)\leq \TC_{G,n}(X)$, we can conclude that the inequality \eqref{eq: eqhtc-gpact-ub} generalizes  \cite[Corollary 4.7]{daundkarlens}.
\end{remark}

\begin{proposition}\label{prop:eqptc-eqcat-ub}
 Let $p:E\to B$ be a $G$-fibration such that the fibre of $p|_{E^H}:E^H\to B^H$ is path connected for all closed subgroups $H$ of $G$ and let $E^n_B$ be $G$-connected. Then \[\TC_{G,n}[p:E\to B]\leq \ct_{G}(E^n_B).\]    
\end{proposition}
\begin{proof}
Since $p|_{E^H}:E^H\to B^H$ is path connected for all closed subgroups $H$ of $G$, the map  $\Pi_n|_{(E^I_B)^H}:(E^I_B)^H\to (E^n_B)^H$ is surjective. Therefore, the conclusion follows from \cite[Proposition 4.5]{EqTC}.
\end{proof}

\begin{proposition}
Let $p:E\to B$ be a $G$-fibration. If $E$ is $G$-contractible space, then 
\[\ct_G(E^n_B)\leq \TC_n[p:E\to B].\]
\end{proposition}
\begin{proof}
 Let $x\in E$ and $\gamma_x$ be a constant path at $x$. Then define $h:E\to E^I_B$ by $h(x)=\gamma_x$. Observe that $h$ is a $G$-map since $h(gx)=\gamma_{gx}=g\cdot \gamma_x=g\cdot h(x)$. Notice that $h$ has a $G$-homotopy inverse $h'$, which sends a path to its initial point. 
 This implies, $E\simeq_{G}E^I_B$. This implies $E^I_B$ is a $G$-contractible space. Therefore, we can apply \cite[Proposition 4.6]{EqTC} to the fibration $\Pi_n:E^I_B\to E^n_B$ to get the desired result. 
\end{proof}

Now, the following result is clear.
\begin{corollary}
 Let $p:E\to B$ be a $G$-fibration and let $E$ be a $G$-contractible space. If $E^n_B$ is $G$-connected and $E^G\neq \emptyset$, then \[\TC_n[p:E\to B]=\ct_G(E^n_B).\]
\end{corollary}

\begin{theorem}\label{thm:fixed-set-ptc-ineq}
 Let $p:E\to B$ be a $G$ fibration. Let $H$ and $K$ be any closed subgroups of $G$ such that $E^H$ is $K$-invariant.
 Then \[\TC_{K,n}[p:E^H\to B^H]\leq \TC_{G,n}[p:E\to B].\]
\end{theorem}
\begin{proof}
Let $U$ be a $G$-invariant open set of $E^n_B$ on which we have a local $G$-section $s$ of $\Pi_n:E^I_B\to E^n_B$.  Let $V=U\cap (E^H)^n_B $. Then, observe that $V$ is $K$-invariant.   Note that for $x\in (E^H)^n_B$, we have $hs(x)=s(hx)=s(x)$. Therefore, $s(x)\in (E^I_B)^H$. But $(E^I_B)^H=(E^H)^I_{B^H}$. Thus, $s$ restricts to $V$. Therefore, we have a $K$-section $s|_{V}:V\to (E^H)^I_B$. This concludes the result. 
\end{proof}

The following is an obvious consequence of \Cref{thm:fixed-set-ptc-ineq}.
\begin{corollary} \label{cor:subgp-htc-ineq}
Let $p:E\to B$ be a $G$-fibration. Then for any closed subgroups $H$ and $K$ of $G$ we have:
\begin{enumerate}
\item $\TC_n[p|_{E^H}:E^H\to B^H]\leq \TC_{G,n}[p:E\to B]$,
\item $\TC_{K,n}[p:E\to B]\leq \TC_{G,n}[p:E\to B]$.
\end{enumerate}
\end{corollary}

\begin{remark} \
   If $B$ is a singleton, then \Cref{thm:fixed-set-ptc-ineq}    generalizes \cite[Proposition 5.3]{EqTC} and \cite[Proposition 3.14]{byehsarkareqtcn}.
\end{remark}

In the following result, we establish the relation between sequential equivariant parametrized topological complexity of a $G$-fibration and the equivariant category of its fibre and the equivariant category of the corresponding fibre product. The proof is inspired by \cite[Proposition 5.7]{EqTC}.

\begin{theorem}
Let $p:E\to B$ be a $G$-fibration such that $E^n_B$ is $G$-connected, and let $H$ be a stabilizer of some $e\in E$. Then, \[\ct_H(F^{n-1})\leq \TC_{G,n}[p:E\to B] \leq \ct_G(E^n_B).\]
In particular, if $E^G\neq \emptyset$, then 
\[\ct_G(F^{n-1})\leq \TC_{G,n}[p:E\to B]\leq \ct_G(E^n_B).\]
\end{theorem}
\begin{proof}
 Let $e\in E$ and $b=p(e)$. We consider $F=p^{-1}(b)$. Note that $b\in B^H$. Therefore, $F$ admits an $H$-action. We now define a map $f:F^{n-1}\to E^n_B$ by $f(x_1,\dots,x_{n-1})=(e,x_1,\dots,x_{n-1})$.
 Note that $f$ is a well-defined $H$-equivariant map. 
 Now consider the following pullback diagram:
\[ \begin{tikzcd}
Y \arrow{r}{} \arrow[swap]{d}{q} & E^I_B \arrow{d}{\Pi_n} \\%
F^{n-1} \arrow{r}{f}&E^n_B\\%
& 
\end{tikzcd}
,\] where $Y=\{\gamma\in E^I_B 
\mid \gamma(0)=e\}$.
Note that $Y$ is $H$-invariant and $H$-contractible. Therefore, $\ct_H(F^{n-1})\leq \sct_{H}(q)$.
From \cite[Proposition 4.3]{EqTC}, we have: \[\sct_H(q)\leq \sct_H(\Pi_n)=\TC_{H,n}[p:E\to B].\] Now from \Cref{cor:subgp-htc-ineq}, we have $\TC_{H,n}[p:E\to B] \leq \TC_{G,n}[p:E\to B]$. This proves the desired inequality.
\end{proof}


We will now obtain the cohomological lower bound on the equivariant parametrized topological complexity. Let $X$ be a $G$-space, and $X^h_G=EG\times_{G} X$ be the homotopy orbit space of $X$. The Borel $G$-equivariant cohomology of $X$ is defined as   $H^{\ast}_G(X;R):=H^{\ast}(X^h_G;R)$, where $R$ is any commutative ring.

Let $p:E\to B$ be a $G$-fibration and $E^n_B$ be the corresponding fibre product. 
Note that the diagonal map $\triangle: E\to E^n_B$ is $G$-equivariant. Therefore, it induces a map $\triangle_G: E^h_G\to (E^n_B)^h_G$.

\begin{theorem}\label{thm:coho-lb}
Let $\triangle_G^{\ast}:H^{\ast}_G(E^n_B)\to H^{\ast}_G(E)$ be the induced map from $\triangle_G: E^h_G\to (E^n_B)^h_G$. Suppose there exists cohomology classes $u_i\in \mathrm{ker}(\triangle_G^{\ast})$ for $1\leq i\leq r$  such that $u_1\cup \dots \cup u_r\neq 0$. Then \[\TC_{G,n}[p:E\to B]> r.\]   
\end{theorem}
\begin{proof}
Consider the following commutative diagram:
\[ \begin{tikzcd}
E \arrow{r}{h} \arrow[swap]{d}{\triangle} & E^I_B \arrow{dl}{\Pi_n} \\%
E^n_B &U_i \arrow[l, "i_{U_i}", hook] \arrow{u}{}\\%
& 
\end{tikzcd}
,\]
where $h$ is a $G$-homotopy equivalence.
Suppose $\TC_{G,n}[p:E\to B]\leq r$. Consider a parametrized motion planning cover $\{U_1,\dots, U_r\}$. Since $h$ is $G$-homotopy equivalence, the inclusion $i_{U_i}:U_i\to E^n_B$ factors through the diagonal $\triangle: E\to E^n_B$ up to $G$-homotopy. Therefore, $i_{U_i}^{\ast}(u_i)=0$ for $1\leq i\leq r$. Hence $u_i$'s are in the image of $H^{\ast}_G(E^n_B, U_i)\to H^{\ast}_G(E^n_B)$ for $1\leq i\leq r$.
Since $ U_i$ covers $E^n_B$, we get that $u_1\cup\dots\cup u_r=0$ by the naturality of cup products. This completes the proof. 
\end{proof}

The upper bound on the equivariant sectional category of Serre $G$-fibrations has obtained in \cite[Theorem 3.5]{Grantsymmtc}. We now state this result with a slight modification.

\begin{theorem}[{\cite[Theorem 3.5]{Grantsymmtc}}\label{thm:ub-gsecat}]
 Let $p:E\to B$ be a Serre $G$-fibration and let $B$ be a $G$-CW complex with $\dim(B)\geq 2$. Suppose for all closed subgroups $H$ of $G$, the fibre of $p^H:E^H\to B^H$ is $(m-1)$-connected. Then
 \[\sct_G(p)< \frac{\dim(B)+1}{m+1}+1.\]   
\end{theorem}
We will now use \Cref{thm:ub-gsecat} to establish the upper bound on the equivariant parametrized topological complexity.


\begin{theorem}\label{thm:dim-conn-ub}
Let $p:E\to B$ be a (Hurewicz) $G$-fibration and  $E^n_B$ is a $G$-CW complex with $\mathrm{dim}(E^n_B)\geq 2$. Suppose for all subgroups $H$ of $G$, the fibre of $p^H:E^H\to B^H$ is $m$-connected.
Then \[\TC_{G,n}[p:E\to B]< \frac{n\mathrm{dim}(F)+\mathrm{dim}(B)+1}{m+1}+1.\]
\end{theorem}
\begin{proof}
Since $\Pi_n:E^I_B\to E^n_B$ is a (Hurewicz) $G$-fibration, it is also a Serre $G$-fibration.  
We can observe that $(E^I_B)^H=(E^H)^I_{B^H}$ and $(E^n_B)^H=(E^H)^n_{B^H}$ for any closed subgroup $H$ of $G$. Thus, the fibration $(\Pi_n)^H:(E^I_B)^H\to (E^n_B)^H$ is, in fact, a sequential parametrized path fibration associated with the fibration $p^H:E^H\to B^H$. In other words, we have $(\Pi_n)^H:(E^H)^I_{B^H}\to (E^H)^n_{B^H}$. Let $X_H$ denote the fibre of $p^H: E^H\to B^H$.
Then, note that the fibre of $(\Pi_n)^H$ is $(\Omega X_H)^{n-1}$.
Since $X_H$ is $m$-connected and $\pi_i((\Omega X_H)^{n-1})=(\pi_i(\Omega X_H))^{n-1}=(\pi_{i+1}(X_H))^{n-1}$, it follows that $(\Omega X_H)^{n-1}$ is $(m-1)$-connected. 
Then from \Cref{thm:ub-gsecat}, we obtain the following inequality:
\begin{equation}\label{eq:eqptcub}
 \TC_{G,n}[p:E\to B]< \frac{\mathrm{dim}(E^n_B)+1}{m+1}+1.   
\end{equation}
It is well known that there is a locally trivial fibration $E^n_B\to B$ with fibre $F^{n}$. Therefore, we have $\mathrm{dim}(E^n_B)\leq n\mathrm{dim}(F)+\mathrm{dim}(B)$. This completes the proof.
\end{proof}

\section{Examples}
\subsection{Equivariant Fadell-Neuwirth fibrations}
In this subsection, we compute sequential equivariant parametrized topological complexity of Fadell-Neuwirth fibrations.

We start with recalling the definition of the ordered configuration space.
The ordered configuration space of $s$ points on $\R^d$ is denoted by $F(\R^d,s)$ and defined as
\[F(\R^d,s)=\{(x_1,\dots,x_s)\in (\R^d)^s \mid x_i\neq x_j ~\text{ for }~ i\neq j\}.\]

\begin{definition}[{\cite{FadellNeuwirth}}\label{def:fadellNeuiwirth}]
The maps \[p: F(\R^d,s+t)\to F(\R^d,s) ~\text{ defined by }~ p(x_1,\dots,x_{s+t})=(x_1,\dots,x_s)\] are called Fadell-Neuwirth fibrations.    
\end{definition}

The sequential parametrized topological complexity of these fibrations have been computed in \cite{SequentialPTC} for odd $d\geq 3$  and  in \cite{FarberPaul2} for even $d\geq 2$.
We now state their theorem.

\begin{theorem}[{\cite[Theorem 1.3]{FarberPaul2}}]
Let $d\geq 2$ be an even integer, and $t\geq 1$, $s\geq 2$. Then 
\[\TC_n[p:F(\R^{d},s+t)\to F(\R^{d},s)]=nt+s-1.\]
\end{theorem}
\begin{theorem}[{\cite[Theorem 8.1]{SequentialPTC}}\label{thm: sptcFadNeu}]
Let $d\geq 3$ be an odd integer, and $s \geq  2$, $t \geq 1$.
Then \[\TC_n[p:F(\R^{d},s+t)\to F(\R^{d},s)]=nt+s.\]
\end{theorem}

We now define an action of a permutation group $\Sigma_s$ (on $s$ letters)  on $F(\R^d,s+t)$. Let $\sigma\in \Sigma_s$. Define
\[\sigma\cdot (x_1,\dots,x_s,y_1,\dots,y_t)=(x_{\sigma(1)},\dots,x_{\sigma(s)},y_1,\dots,y_t).\]
Also $\Sigma_s$ acts on $F(\R^d,s)$ by permuting coordinates. Observe that the map $p$ in \Cref{def:fadellNeuiwirth} 
satisfies \[p((x_{\sigma(1)},\dots,x_{\sigma(s)},y_1,\dots,y_t))=(x_{\sigma(1)},\dots,x_{\sigma(s)})=\sigma\cdot p((x_1,\dots,x_s,y_1,\dots,y_t)).\]
Thus, $p$ is a $\Sigma_s$-equivariant map. We now show that $p$ is indeed a $\Sigma_s$-fibration.

The local triviality of $p$ was shown in  \cite[Theorem 1.1]{confspacefadell}, and we also refer the reader to \cite[Theorem 3.3]{FredCohenconfg}.
We will use the ideas from these results to show the $\Sigma_s$-equivariant local triviality of $p$. 
Consider a point $q=(q_1,\dots,q_s)\in F(\R^d,s)$ and mutually disjoint open discs $D_i$ with centres $q_i$. 
Then consider the open set $V=\cup_{\sigma\in \Sigma_s}D_{\sigma(1)}\times \dots\times D_{\sigma(s)}$ containing $q$. Note that $V$ is $\Sigma_s$-invariant and a disjoint union.
Note that the fibre $F_q$ of $p$ at $q$ is given by 
$F(\R^d\setminus \{q_1,\dots,q_s\},t)$. We identify $F_q$ with the last $t$-coordinates of $F(\R^d,s+t)$ via $(y_1,\dots,y_t)\mapsto (q_1,\dots,q_s,y_1,\dots,y_t)$ and assume that $\Sigma_s$ acts trivially on $F_q$. 
We define a homeomorphism $\Phi_q:V\times F_q\to p^{-1}(V)$ in a similar way as it was defined in \cite[Theorem 1.1]{confspacefadell}.
We fix the notation $\bar{y}=(y_1,\dots,y_t)$.
Let $\bar{x}=(x_1,\dots,x_s)\in V$. That means $\bar{x}\in D_{\sigma(1)}\times \dots\times D_{\sigma(s)}$ for some $\sigma\in \Sigma_s$. 
Define \[\Phi_q(x_1,\dots,x_s,\bar{y})=
    (x_1,\dots,x_s,(\gamma_{\sigma(1)}^{-1}(x_1)\circ\dots\circ \gamma_{\sigma(s)}^{-1}(x_s))(\bar{y})),\]
where functions $\gamma_i(x_i):\R^d\to \R^d$ are homeomorphisms defined as in the \cite[Corollary 1.1]{confspacefadell} for $1\leq i\leq s$ and we have denoted $(\gamma_{i}(x))^{-1}$ by $\gamma_i^{-1}(x)$. Recall that these functions have the following properties: for $x\in D_i$, we have $\gamma_i(x)(D_i)=D_i$, $\gamma_i(x)(x)=q_i$, and $\gamma_i(x)(y)=y$ for $y\in \R^d\setminus D_i$.
Note that \[(\gamma_1^{-1}(x_1)\circ\dots\circ \gamma_s^{-1}(x_s))(\bar{y})=((\gamma_1^{-1}(x_1)\circ\dots\circ \gamma_s^{-1}(x_s))(y_1),\dots, (\gamma_1^{-1}(x_1)\circ\dots\circ \gamma_s^{-1}(x_s))(y_t)).\]
For $(x_1,\dots,x_s)\in D_1\times\dots\times D_s$, we note the important observation:
\begin{equation}\label{eq:gamma-composition}
(\gamma_1^{-1}(x_1)\circ\dots\circ \gamma_s^{-1}(x_s))(y_i)=\begin{cases} \gamma_j^{-1}(x_j)(y_i),& \text{ if } y_i\in D_j,\\
y_i,& \text{otherwise}.    
\end{cases}
\end{equation}
We will now prove that $\Phi_q$ is $\Sigma_s$-equivariant.
Let $\bar{x}\in D_{\sigma(1)}\times\dots\times D_{\sigma(s)}$ and $\rho\in\Sigma_s$.
Note that $\rho\cdot \bar{x}\in D_{\sigma(\rho(1))}\times\dots\times D_{\sigma(\rho(s))}$. Then we have
\[\Phi_q(\rho\cdot \bar{x},\bar{y})=(\rho\cdot \bar{x}, (\gamma_{\sigma(\rho(1))}^{-1}(x_{\rho(1)})\circ\dots\circ \gamma_{\sigma(\rho(s))}^{-1}(x_{\rho(s)}))(\bar{y})).\]
Observe that the sets $\{\gamma_{\sigma(r)}^{-1}(x_r) \mid 1\leq r\leq s\}$ and $\{\gamma_{\sigma(\rho(r))}^{-1}(x_{\rho(r)}) \mid 1\leq r\leq s\}$ coincides. Then using \eqref{eq:gamma-composition} we can see that 
\[(\gamma_{\sigma(\rho(1))}^{-1}(x_{\rho(1)})\circ\dots\circ \gamma_{\sigma(\rho(s))}^{-1}(x_{\rho(s)}))(y_i)=(\gamma_{\sigma(1)}^{-1}(x_1)\circ\dots\circ \gamma_{\sigma(s)}^{-1}(x_s))(y_i).\] 
More precisely, if $y_i\in D_j$ for some $1\leq j \leq s$ and $\sigma(k)=j$, then we have \[(\gamma_{\sigma(1)}^{-1}(x_1)\circ\dots\circ \gamma_{\sigma(s)}^{-1}(x_s))(y_i)=\gamma^{-1}_{\sigma(k)}(x_k)(y_i).\]
Moreover, if $\sigma(\rho(r))=j$, then 
\[(\gamma_{\sigma(\rho(1))}^{-1}(x_{\rho(1)})\circ\dots\circ \gamma_{\sigma(\rho(s))}^{-1}(x_{\rho(s)}))(y_i)=\gamma^{-1}_{\sigma(\rho(r))}(x_{\rho(r)})(y_i).\]
Since $\rho(r)=k$, we get $\gamma^{-1}_{\sigma(k)}(x_k)(y_i)=\gamma^{-1}_{\sigma(\rho(r))}(x_{\rho(r)})(y_i)$.
This gives us, \[\Phi_q(\rho\cdot \bar{x},\bar{y})=(\rho\cdot \bar{x}, (\gamma_{\sigma(1)}^{-1}(x_1)\circ\dots\circ \gamma_{\sigma(s)}^{-1}(x_s))(\bar{y}))=\rho\cdot(\bar{x}, (\gamma_{\sigma(1)}^{-1}(x_1)\circ\dots\circ \gamma_{\sigma(s)}^{-1}(x_s))(\bar{y})).\]
This proves that $\Phi_q$ is $\Sigma_s$-equivariant.

Let $F=F(\R^d\setminus \{a_1,\dots,a_s\},t)$ be the fibre of $p$ with $(a_1,\dots,a_s)\in F(\R^d,s)$ as chosen base point.
We will now choose a stable homeomorphism (as in step-$2$ of \cite[Theorem 1.1]{confspacefadell}) $\alpha_q:F\to F_q$ and define $\Phi=\Phi_q\circ (id_V\times \alpha_q)$. Note that $\alpha_q$ is equivariant as $\Sigma_s$ acts trivially on both $F$ and $F_q$. We conclude that $\Phi:V\times F\to p^{-1}(V)$ is $\Sigma_s$-equivariant local trivialization.  
Finally, using \cite[Theorem 2.3]{eqliftingfn}, we conclude that the fibration $p$ is a $G$-fibration.

We are going to use \Cref{thm:dim-conn-ub} to obtain an upper bound on $\TC_{\Sigma_s,n}[p:F(\R^d,s+t)\to F(\R^d, s)]$. 
We denote the fibration $p:F(\R^d,s+t)\to F(\R^d, s)$ by $p:E\to B$.
  Note that $F$ is $(d-2)$-connected.
Since any closed subgroup $H$ of $\Sigma_s$ acts freely on $E$ and $B$, we only need to consider $p^H:E^H\to B^H$ when $H$ is the trivial subgroup. In this case, $p^H=p$. 
Then from \Cref{thm:dim-conn-ub}, we get the following expression.
\[\TC_{\Sigma_s,n}[p:E\to B]< \frac{(d-1)(nt+s-1) +1}{d-1}+1=nt+s + 1/(d-1).\]
This gives us  
\[\TC_{\Sigma_s,n}[p:F(\R^d,s+t)\to F(\R^d, s)]\leq nt+s.\]
Since we have \[\TC_n[p:F(\R^d,s+t)\to F(\R^d, s)]\leq \TC_{\Sigma_s,n}[p:F(\R^d,s+t)\to F(\R^d, s)],\] we establish the following result.


\begin{theorem}\ 
\begin{enumerate}
    \item Let $d\geq 3$ be an odd integer, and $s\geq 2$, $t\geq 1$. Then     \[\TC_{\Sigma_s,n}[p:F(\R^d,s+t)\to F(\R^d, s)]=nt+s.\]
    \item Let $d\geq 2$ be an even integer, and $t\geq 1$, $s\geq 2$. Then 
    \[\TC_{\Sigma_s,n}[p:F(\R^d,s+t)\to F(\R^d, s)]\in \{nt+s-1,nt+s\}.\]
\end{enumerate}
\end{theorem}

\subsection{Some generalized projective product spaces}
Let $\tau$ be any involution on $M$. Consider the following generalized projective product space \[
X(M, S^{2m+1}) :=\displaystyle\frac{M \times S^{2m+1}}{(x,y)\sim (\tau(x), -y)}.
\]
We have a fibre bundle $p:X(M, S^{2m+1})\to \R P^{2m+1}$ with fibre $M$.
We define a $\Z_2$-action on $X(M, S^{2m+1})$ by defining an involution \[\sigma([x,y]):=[x,\tilde{y}],\]
where $y=(y_0,y_1,\dots,y_{2m},y_{2m+1})$ and $\tilde{y}=(-y_1,y_0,\dots, -y_{2m+1},y_{2m})$.
It can be observed that $\sigma$ defines a free involution on $X(M, S^{2m+1})$.
A $\Z_2$-action on $\R P^{2m+1}$ is defined by 
\[g\cdot [y_0:y_1:\dots:y_{2m}:y_{2m+1}]:=[-y_1:y_0:\dots: -y_{2m+1}:y_{2m}].\]
We now observe the following: 
\[p([x,(-y_1,y_0,\dots, -y_{2m+1},y_{2m})])=[-y_1:y_0:\dots: -y_{2m+1}:y_{2m}]=g\cdot [y_0:y_1:\dots:y_{2m}:y_{2m+1}].\]
This shows that $p:X(M,S^{2m+1})\to \R P^{2m+1}$ is a $\Z_2$-equivariant map. One can also show that $p$ is in fact a $G$-fibration.
Observe that \[X(M,S^{2m+1})^n_{\R P^{2m+1}}=\{([x_1,y_1],\dots,[x_n,y_n])\in X(M,S^{2m+1})^n \mid y_i=\pm y_j, 1\leq i,j\leq n\}.\]
Since $\Z_2$ acts freely on $X(M,S^{2m+1})$, we have free diagonal $\Z_2$-action on $X(M,S^{2m+1})^n_{\R P^{2m+1}}$. Thus $X(M,S^{2m+1})^n_{\R P^{2m+1}}$ is $\Z_2$-connected. Therefore using \Cref{prop:eqptc-eqcat-ub} we get the following inequality:
\[\TC_{\Z_2,n}[p:X(M,S^{2m+1})\to \R P^{2m+1}]\leq \ct_{\Z_2}(X(M,S^{2m+1})^n_{\R P^{2m+1}}).\]
From \cite[Proposition 3.5]{EqTC} we get $\ct_{\Z_2}(X(M,S^{2m+1})^n_{\R P^{2m+1}})=\ct(X(M,S^{2m+1})^n_{\R P^{2m+1}}/\Z_2)$. 
Using the dimensional inequality for category, we obtain \[\TC_{\Z_2,n}[p:X(M,S^{2m+1})\to \R P^{2m+1}]\leq n\dim(M)+2m+2.\]
In particular, if $M=\R P^k$ with any involution, then we have \[\TC_n(\R P^k)\leq \TC_{\Z_2,n}[p:X(\R P^k,S^{2m+1})\to \R P^{2m+1}]\leq nk+2m+2.\]
Suppose $k=2^r$, then using \cite[Corollary 8.2]{FarberTCproj}, we have 
\[nk\leq \TC_{\Z_2,n}[p:X(\R P^k,S^{2m+1})\to \R P^{2m+1}]\leq nk+2m+2.\]



\section{Acknowledgement} 
The author expresses sincere gratitude to the reviewer for numerous insightful suggestions, comments, and feedback that significantly improved the exposition and presentation of this article.
In particular, the author thanks the reviewer for prompting the proof that $\Pi_n$ is a $G$-fibration, for suggesting the correct version of \Cref{thm:ub-gsecat} and the modifications in the proof of \Cref{thm:dim-conn-ub}, and for many other valuable suggestions.
The author also thanks Soumen Sarkar for valuable feedback on an earlier version of the article, and Mark Grant for his valuable suggestions and for providing resources related to equivariant principal bundles, which were instrumental in proving \Cref{thm:eqptc-eq-prin-bundle}.  Additionally, the author appreciates Tejas Kalelkar and Ramandeep Singh Arora for helpful discussions. The author acknowledges the support of NBHM through grant 0204/10/(16)/2023/R\&D-II/2789.

.

\bibliographystyle{plain} 
\bibliography{references}

\end{document}